\documentclass[12pt,reqno]{amsart}
\usepackage{amssymb,latexsym}
\usepackage{graphicx}
\usepackage{amsmath,amsthm}
\usepackage{enumerate}
\usepackage{url}		

\newcommand{\F}{{\mathbb F}}
\newcommand{\K}{{\mathbb K}}

\newcommand{\R}{{\mathbb R}}
\newcommand{\Q}{{\mathbb Q}}
\newcommand{\Z}{{\mathbb Z}}

\newcommand{\ta}{{\theta }}

\newcommand{\sqf}{{\rm sqf}}
\newcommand{\rank}{{\rm rank}}
\newcommand{\Gal}{{\rm Gal}}

\theoremstyle{plain}
\numberwithin{equation}{section}
\newtheorem{thm}{Theorem}[section]
\newtheorem{theorem}[thm]{Theorem}

\newtheorem{example}[thm]{Example}

\newtheorem{proposition}[thm]{Proposition}

\newtheorem{remark}[thm]{Remark}

\begin{document}

\setcounter{page}{1}

\title[On $\ta$-congruent numbers]{On $\ta$-congruent numbers on real quadratic number fields}
\author{Ali S. Janfada}
\address{Department of Mathematics\\
                Urmia University\\
                Urmia\\
                Iran}
\email{a.sjanfada@urmia.ac.ir,   asjanfada@gmail.com}
\thanks{}
\author{Sajad Salami}
\address{Instituto da Mathemçática  e Estatistica,
         UERJ, Brazil}
\email{sajad.salami@ime.uerj.br}
\begin{abstract}
Let $\K=\Q(\sqrt{m})$ be a real quadratic number field, where $m>1$ is a squarefree integer.
Suppose that $0 < \ta< \pi$ has rational cosine, say $\cos (\ta)=s/r$ with $0< |s|<r$ and $\gcd(r,s)=1$. 
A positive integer $n$ is called a $(\K,\ta)$-congruent number 
if there is a triangle, called the $(\K,\ta, n)$-triangles, with sides in $\K$ having $\ta$ as an angle 
and $n\alpha_\ta$ as  area, where ${\alpha_\ta}=\sqrt{r^2-s^2}$. 
Consider the $(\K,\ta)$-congruent number 
elliptic curve $E_{n,\ta}: y^2=x(x+(r+s)n)(x-(r-s)n)$ defined over $\K$. 
Denote the squarefree part of  positive integer $t$ by $\sqf(t)$. 
In this work, it is proved that if $m\neq \sqf(2r(r-s))$ and $mn\neq  2, 3, 6$,  
then $n$ is a $(\K,\ta)$-congruent number  if and only 
if the Mordell-Weil group $E_{n,\ta}(\K)$
has positive rank, and  all of the  $(\K,\ta, n)$-triangles are classified in four types. 
\end{abstract}
\maketitle
\par {\bf Subject class [2010]: }{Primary 11G05, Secondary 14H52}

\section{Introduction}
\label{Intro}
A positive integer $n$ is called a {\it congruent number} if it 
is the area of a right triangle with rational sides.
Finding all  congruent numbers is one of the classical problems 
in the modern number theory.  
We cite \cite{kob1} for an  exposition of the congruent number problem,
and \cite{fujw1} to see the first study of  $\ta$-congruent numbers as a generalization of the classic one.
Let $0 < \ta< \pi$ has rational cosine $\cos (\ta)=s/r$ with $0< |s|<r$ 
and $\gcd(r,s)=1$. Let $(U,V,W)_\ta$ denote a triangle with an angle $\ta$ 
between sides $U$ and $V$. A positive integer $n$ is called a $\ta$-{\it congruent number} if 
there exists a  triangle $(U,V,W)_\ta$ with  sides in $\Q$ having area  
$n\alpha_\ta$, where ${\alpha_\ta}=\sqrt{r^2-s^2}$. In other words, $n$ is a $\ta$-congruent number if it satisfies
$$2rn=UV, \ \ W^2=U^2+V^2 -\frac{2s}{r}UV.$$
An ordinary  congruent number is  nothing but  
a $\pi/2$-congruent number. Clearly,  if  $n$ is a $\ta$-congruent number, then so is $nt^2$, for any positive integer $t$.  We shall concentrate on squarefree numbers whenever $\ta$-congruent numbers concerned. 
Let 
 $$E_{n,\ta}: \ y^2=x(x+(r+s)n)(x-(r-s)n)$$ 
be the $\ta$-congruent number elliptic curve, where $r$ and $s$ are  as above. 
Theorem \ref{fuj1}  gives an important connection between  $\ta$-congruent numbers and the Mordell-Weil group $E_{n,\ta}(\Q)$. 
For more information and recent results about $\ta$-congruent numbers see \cite{fujw2,ajs,SY1}.

The notion $\ta$-congruent number, which is defined over $\Q$, can be extended in a natural way 
over real quadratic number fields $\K$. In this case, we refer to $n$ as a {\it $(\K,\ta)$-congruent number}
and to the triangle $(U,V,W)_\ta$ as a  $(\K,\ta,n)$-{\it triangle}. 
When $n$ is not a $\ta$-congruent number over $\Q$, a  question  proposed naturally:
{\it Is $n$  a $(\K,\ta)$-congruent number for  some real quadratic number field $\K$?}
Tada \cite{tada} answered this question in the case $\ta=\pi/2$, by studying the structure of the $\K$-rational points on the elliptic curve $E_{n, \pi/2}: y^2=x(x^2-n^2)$. In this paper,  we answer the above question  for any$0 <\ta <\pi$ and  classify all $(\K,\ta,n)$-triangles. 
Through the paper we shall consider  $\K=\Q(\sqrt{m})$ to be a real quadratic field, where  $m>1$ is  squarefree. We denote the squarefree part of any positive integer $N$ by $\sqf(N)$.
The main results of this paper are the following theorems. 
\begin{theorem}
\label{th1}
Let $n$ be a positive squarefree integer with $\gcd(m,n)=1$  
such that $mn\neq 2,3,6$  and $m\neq \sqf(2r(r-s))$, where $m,r,s$ are as before. Then $n$ is a $(\K,\ta)$-congruent 
number if and only if $\rank(E_{n,\ta}(\K))>0$. Moreover, $n$ is a $(\K,\ta)$-congruent number  if and only if either $n$ 
or $mn$ is a $\ta$-congruent number over $\Q$.
\end{theorem}
Theorem \ref{th1} is an extension of Part (2) of Theorem \ref{fuj1} in the following. 
Note that the non-equality conditions for $mn$ and $m$ in Theorem \ref{th1} are necessary. 
For a counterexample, when $n=1$ and $\ta= 2\pi/3$, we have $r=2$, $s=-1$, 
$\alpha_{\ta} =\sqrt{3}$. Now taking $m=3=\sqf(2r(r-s))$, there is a 
$(\Q(\sqrt{3}), \ta,1)$-triangle with sides $(2,2,2\sqrt{3})$ and area $\sqrt{3}$ 
but using  Theorem \ref{A},  $\rank(E_{1,\ta}(\Q(\sqrt{3})))= \rank(E_{1,\ta}(\Q)) +\rank(E_{3,\ta}(\Q))=0.$

The following theorem classifies all types of $(\K,\ta,n)$-triangles.
\begin{theorem}
\label{th2}
Assume that $n$ is not  a $\ta$-congruent number over $\Q$ and let $\sigma$ be 
the generator of  $\Gal(\K/\Q)$.  Then 
any $(\K,\ta,n)$-triangle with $(U,V,W) \in (\K^*)^3$ and $(0<U\leq V < W )$
 is necessarily one of the following types:
 \begin{enumerate}[\upshape {\it Type} 1.]
 \item $U \sqrt{m}$, $V \sqrt{m}$, $W \sqrt{m} \in \Q$;
 \item $U$, $V$, $W \sqrt{m} \in \Q$;
 \item $U$, $V \in \K\backslash \Q$ such that $\sigma(U)=V$, $W\in \Q$; 
 \item $U$, $V \in \K\backslash \Q$ such that $\sigma(U)=-V$, $W\in \Q$.
 \end{enumerate}
  \end{theorem}
Let $A=\sqf (r^2-s^2)$, $B=\sqf (2r(r-s))$ and $C=\sqf (2r(r+s))$. The following proposition shows when there is no $(\K,\ta,n)$-triangle of Types $2$, $3$ and $4$.
\begin{proposition}\label{L1}
Let $p$ be a prime number and the pair $(m, A)$ (resp. $(m, B)$ and $ (m, C )$)   can  be written as 
$(p^\alpha a, p^\beta b )$, where $\alpha, \beta \in \{0,1 \}$ and $gcd(p, ab)=1$. 
 Then there is no $(\K,\ta,n)$-triangle of Type 2 (resp. Type 3 and Type 4) whenever one of the following conditions hold.
 \begin{enumerate}[\upshape  (1)]
  \item $\underline{p=2:}$  
      $(\alpha, \beta)=(0,0)$ and $(a,b) \mathop   \equiv \limits^{4 }  (3,3)$, \\ 
     $(\alpha, \beta)=(0,1)$ and $(a,b) \mathop   \equiv \limits^{8 }  
     (3,1), (3,5), (7,5), (7,7)$, \\
     $(\alpha, \beta)=(1,0)$ and $(a,b) \mathop   \equiv \limits^{8 } 
       (1,3), (1,5), (3,5), (3,7), (5,3), (5,7),(7,3), (7,7)$,\\
      $(\alpha, \beta)=(1,1)$ and  $(a,b) \mathop   \equiv \limits^{8 } 
      (1,3), (1,5), (3,1), (3,3),(5,1), (5,7),(7,5), (7,7)$;
 \item   $\underline{p \mathop \equiv \limits^{4 } 1 :}$ 
    $(\alpha, \beta)=(0,1)$ and $\big(\frac{a}{p}\big)=-1$, \ 
     $(\alpha, \beta)=(1,0)$ and $\big(\frac{b}{p} \big)=-1$, \\
     $(\alpha, \beta)=(1,1)$ and $\big(\frac{a}{p} \big)\big(\frac{b}{p} \big)=-1$;
  \item  $\underline{p \mathop \equiv \limits^{4 } 3 :}$  
       $(\alpha, \beta)=(0,1)$ and $\big(\frac{a}{p} \big)=-1$, \
			 $(\alpha, \beta)=(1,0)$ and $\big(\frac{b}{p} \big)=-1$, \\
      $(\alpha, \beta)=(1,1)$ and $\big(\frac{a}{p} \big)\big(\frac{b}{p} \big)=1$.
   \end{enumerate}
\end{proposition}
\par The next result settles a condition on $n$ and $mn$ to be $\ta$-congruent  over $\Q$.
\begin{theorem}
\label{th3}
Let  $n$ be a positive squarefree integer such that $\gcd(m,n)=1$ 
and $mn\neq 2,3,6$. Then the  following statements are equivalent.
\begin{enumerate}[\upshape (1)]
\item There is a $(\K,\ta,n)$-triangle $(U,V,W)_\ta$ with $0<U\leq V < W$, $W \not \in \Q$
  and $W\sqrt{m} \not \in \Q$;
\item The integers $n$ and $mn$ are $\ta$-congruent numbers over $\Q$.
\end{enumerate}
\end{theorem}
\section{Preliminaries}
\label{sec1}
Consider an elliptic curve $E: y^2=x^3+ax^2+bx+c$  over $\Q$.
Recall that the $m$-twist $E^m$ of $E$ is an elliptic curve over $\Q$ defined by $y^2=x^3+amx^2+bm^2x+cm^3$. 
The next result establishes a fact about  ranks \cite{serf}.
\begin{theorem}\label{A}
Let  $E$ be an elliptic curve over $\Q$. Then
$${\rm rank}(E(\K))={\rm rank}(E(\Q))+{\rm rank}(E^m(\Q)).$$
\end{theorem}
We denote the torsion subgroup of the  groups $E(\K)$ and $E^m(\K)$ by $T(E,\K)$ and $T(E^m,\K)$, respectively.
Also, we write  $T_{n,\ta}(\K)$ and  $T_{n,\ta}^m(\K)$, respectively, in the case $E=E_{n,\ta}$. 
The following proposition  and theorem have essential roles in the proof of our results.
\begin{proposition}[{\cite[Proposition 1]{kwon}}]
\label{B} 
Let $E$ be an elliptic curve over $\K$.
Then the map  $$\phi: T(E, \K)/T(E,\Q)\rightarrow T(E^m, \Q), \ \ \ \phi (\tilde{P}):= P-\sigma(P)$$
is an injective map of abelian groups, where $\sigma$ is the 
generator of $\Gal(\K/\Q)$.
\end{proposition}
\begin{theorem}[{\cite[Theorem 4.2]{knap}}] 
\label{C}
Let $\F$ be an algebraic number field and  $E$  an elliptic curve 
over $\F$ defined by 
$$y^2=(x-\alpha_1)(x-\alpha_2)(x-\alpha_3), \  \alpha_1, \alpha_2, 
\alpha_3 \in \F.$$
Suppose that $(x_0,y_0)$ be an $\F$-rational point of $E$. Then, there exists an  
$\F$-rational point $(x_1, y_1)$  with
  $2(x_1, y_1)=(x_0,y_0)$ if and only if  $x_0-\alpha_1,$ $x_0-\alpha_2,$ $x_0-\alpha_3$ 
  are squares in $\F$. 
\end{theorem}
\par The next results give important information about $\ta$-congruent numbers over $\Q$. 
\begin{theorem} 
{\rm (Fujiwara, \cite{fujw1})}
\label{fuj1}
Consider $0< \ta< \pi$ with rational cosine.
\begin{enumerate}[\upshape (1)]
\item A positive integer $n$ is a $\ta$-congruent number if and 
only if $E_{n,\ta}(\Q)$ has a point of order greater than $2$;
\item  If $n\neq 1,2,3,6$, then $n$ is a $\ta$-congruent 
number if and only if $E_{n,\ta}(\Q)$ has positive rank.
\end{enumerate}
\end {theorem}
All possibilities for  the torsion subgroup of  $E_{n,\ta}(\Q)$ can be found in the next result. 
\begin{theorem} {\rm (Fujiwara, \cite{fujw2})}
\label{fuj2}
Let $T_{n,\ta}(\Q)$ be the torsion subgroup of the $\ta$-congruent 
number elliptic curve $E_{n,\ta}$ over $\Q$. 
\begin{enumerate}[\upshape (1)]
	\item $T_{n,\ta}(\Q)\cong \Z_2\oplus \Z_8$ if and only if there exist 
	integers $a,b>0$ such that $\gcd(a,b)=1$, $a$ and $b$ have opposite parity  
	and satisfy either of the following conditions.
	\begin{enumerate}[\upshape (i) ] 
  	\item $n=1$, $r=8a^4b^4$, $r-s=(a-b)^4$, $(1+\sqrt{2})b>a>b,$
	  \item $n=2$, $r=(a^2-b^2)^4$, $r-s=32a^4b^4$, $a>(1+\sqrt{2})b;$
	\end{enumerate}
	\item $T_{n,\ta}(\Q)\cong \Z_2\oplus \Z_6$ if and only if there exist 
	integers $u,v>0$ such that $\gcd(u,v)=1$, $u>2v$ and satisfy one of 
	the following conditions:
	\begin{enumerate}[\upshape (i) ] 
  	\item $n=1$, $r=\frac{1}{2}(u-v)^3(u+v)$, $r+s=u^3(u-2v),$
	  \item $n=2$, $r=(u-v)^3(u+v)$, $r+s=2u^3(u-2v),$
	  \item $n=3$, $r=\frac{1}{6}(u-v)^3(u+v)$, $r+s=\frac{1}{3}u^3(u-2v),$
	  \item $n=6$, $r=\frac{1}{3}(u-v)^3(u+v)$, $r+s=\frac{2}{3}u^3(u-2v);$
	\end{enumerate}
  \item $T_{n,\ta}(\Q)\cong \Z_2\oplus \Z_4$ if and only if either of the following holds. 
  \begin{enumerate}[\upshape (i) ] 
  	\item $n=1$, $2r$ and $r-s$ are squares but not satisfy {\rm (i)} of Part {\rm (1)},
	  \item $n=2$, $r$ and $2(r-s)$ are squares but not satisfy  {\rm (ii)} of Part {\rm (1)};
	\end{enumerate}
	\item Otherwise, $T_{n,\ta}(\Q)\cong \Z_2\oplus \Z_2$.
\end{enumerate}
\end {theorem}
\begin{remark}
\label{rem1}
For any squarefree integer $m>1$,  the $m$-twist $E_{n,\ta}^m$  of the 
elliptic curve $E_{n,\ta}$ is defined by 
$y^2=x(x+(r+s)mn)(x-(r-s)mn)$ which is  equal to $E_{mn,\ta}$, as seen. 
Therefore $E_{n,\ta}^m(\Q)=E_{mn,\ta}(\Q)$, and hence $T_{n,\ta}^m(\Q)=T_{mn,\ta}(\Q)$.
\end{remark}
\section{Proofs}
\label{sec2}
Appealing to Proposition \ref{B}, we first settle all possibilities for  the torsion subgroup of $E_{n,\ta}(\K)$.
Let   $h, k,  $ and $ d $ be integers such that  $2r=h^2\sqf(2r), r-s=k^2\sqf(r-s)$ and 
 $2r(r-s)=d^2m$, where $m=\sqf(2r(r-s))$.
%
\begin{proposition}\label{p1}
 Assume that $m>1$ and  $n$ are  squarefree positive integers such that $\gcd(m,n)=1$ 
 and $mn\neq  2, 3,6$. Let $T_{n,\ta}(\K)$ be the torsion  subgroup of  $E_{n,\ta}(\K)$. 
\begin{enumerate}[\upshape (1)]
\item If $m=\sqf(2r(r-s))$ and $n=\sqf(2r)$, then
 \begin{multline*} T_{n,\ta}(\K)  =   \{ \infty, (0,0), ( -(r+s)n,0 ) , ( (r-s)n,0 ) , \\
                     ( (nh)^2 - nd\sqrt{m}, \pm ( \frac{d^2mn}{h} - n^2hd\sqrt{m}) ) , 
                     ( (nh)^2 + nd\sqrt{m}, \pm ( \frac{d^2mn}{h} + n^2hd\sqrt{m})  ) \}; 
\end{multline*}
\item If $m=\sqf(2r(r-s))$ and $n=\sqf(r-s)$, then
 \begin{multline*} T_{n,\ta}(\K)  =   \{ \infty, (0,0),(-(r+s)n,0 ), ((r-s)n,0 ), \\
                    ( (nk)^2 - nd\sqrt{m}, \pm (\frac{d^2mn}{k} - n^2kd\sqrt{m}) ) ,
                     ( (nk)^2 + nd\sqrt{m}, \pm (\frac{d^2mn}{k} + n^2kd\sqrt{m})  )   \}; 
\end{multline*}										
\item Otherwise, $ T_{n,\ta}(\K)=\{ \infty, (0,0),(-(r+s)n,0 ), ((r-s)n,0 )\}$.
\end{enumerate}
\end{proposition}
\begin{proof}
The   $2$-torsion subgroup of $E_{n,\ta}(\K)$ is:
 $$E_{n,\ta}[2](\K)= \{\infty, (0,0), (-(r+s)n,0), ((r-s)n,0) \}.$$ 
Therefore, we have 
$T_{n,\ta}(\K) \supset E_{n,\ta}[2](\K) \cong \Z/2\Z \oplus \Z/2\Z.$
By Remark \ref{rem1} and Theorem \ref{fuj2}, 
$T_{n,\ta}^m(\Q)=T_{mn,\ta}(\Q)\cong \Z/2\Z \oplus \Z/2\Z$.  
Since $T_{n,\ta}(\Q)\cong \Z/2\Z \oplus \Z/2\Z$,  by Proposition 
\ref{B} and {\cite[Theorem 1]{kwon}} we have
$$T_{n,\ta}(\K)\cong  \Z/2\Z \oplus \Z/2\Z \ {\rm or} \ \Z/2\Z \oplus \Z/4\Z. $$
First let $T_{n,\ta}(\K)\cong \Z/2\Z \oplus \Z/4\Z$. Then there exists 
a point $P=(x_0,y_0)$ of order $4$ in $T_{n,\ta}(\K)$. Then $2P$ must be one of 
the points $(0,0)$, $(-(r+s)n,0)$ and $((r-s)n,0)$. If $2P=(0,0)$ then both $(r+s)n$ and $-(r-s)n$ are squares in $\K$, which is impossible since $\K$  is a real quadratic number field and hence $-1$ is not a square in $ \K $. Similarly, if $2P=(-(r+s)n,0)$, then $-(r+s)n$ and $-2rn$ are squares in $\K$, again a contradiction by the same reason. 
If  $2P=((r-s)n,0)$, then $(r-s)n$ and $2rn$ are squares in $\K$.
Since $n$ is  squarefree, these integers are squares in $\K$ if $m=\sqf(2r(r-s))$.
By a simple computation using the duplication formula  we obtain (1) and (2).
Now,  $T_{n,\ta}(\K)\cong  \Z/2\Z \oplus \Z/2\Z$ implies (3), and the proof is completed.
 \end {proof}
\begin{proof}[\bf{Proof of Theorem \ref{th1}}]
Consider the two sets 
 $$S= \big\{(U,V,W) \in (\K^*)^3: \ 0< U \leq V <W, \ UV=2rn \text{ and } U^2+V^2 -2sUV/r=W^2\big\},$$ 
$$T=\{(u,v)\in 2E_{n,\ta}(\K)\backslash \{\infty\}: \ v\geq0\}.$$
There is a one to one correspondence between the two sets $S$ and $T$ via the two
mutually inverse maps $\varphi : S \rightarrow T$ and $\psi : T \rightarrow S $  defined by
$$\varphi(U,V,W):=( {W^2}/{4},{W(V^2-U^2)}/{8}),$$ 
 $$\psi (u,v):=\big(\sqrt{u+(r+s)n} - \sqrt{u-(r-s)n},
 \sqrt{u+(r+s)n} + \sqrt{u-(r-s)n},2\sqrt{u} \big).$$
 Clearly,  $E_{n,\ta}(\K)\backslash E_{n,\ta}[2](\K) \neq \emptyset$ if
  and only if $S\neq \emptyset$.
 \par Suppose that $m\neq \sqf(2r(r-s))$ and $mn\neq  2, 3,6$. Then by   proposition \ref{p1}, we have  
 $ T_{n,\ta}(\K)= E_{n,\ta}[2](\K)$.
 Therefore, $\rank(E_{n,\ta}(\K))>0$  if and only if 
 $E_{n,\ta}(\K)\backslash E_{n,\ta}[2](\K) \neq \emptyset$.
 So $\rank(E_{n,\ta}(\K))>0$ if and only if 
 either $\rank(E_{n,\ta}(\Q))>0$ or  
 $\rank(E_{n,\ta}^m(\Q))>0$, by Theorem \ref{A}.
 The second part of the theorem follows from 
Remark \ref{rem1}.
\end{proof}
\begin{proof}[\bf{Proof of Theorem \ref{th2}}]
Assume $n$ is a $(\K,\ta)$-congruent 
number and $(U,V,W)_\ta$ is the corresponding  $(\K,\ta,n)$-triangle 
with area $n\alpha_\ta$ such that $0< U\leq V <W$. As  
in the proof of the Theorem \ref{th1}, there is a point $P=(x,y)$ in 
$E_{n,\ta}(\K)\backslash E_{n,\ta}[2](\K)$ such that $\psi(P)=(U,V,W)$.
Substituting $P$ by $P+(0,0)$, $P+(-(r+s)n,0)$ or $P+((r-s)n,0)$, 
if necessary, we may assume that $x>[(r+s)+\sqrt{2r(r-s)}]n$.
 Putting $2P=(u,v)$ and using the map $\psi$ in the proof of Theorem \ref{th1}, we obtain
$$U={2rnx}/{|y|}, \ V={x^2+2snx-(r^2-s^2)n^2}/{|y|}, \
 W={x^2+(r^2-s^2)n^2}/{|y|},$$
 where $x,y \in \K$ and  $|\cdot |$ is the usual absolute value induced from the embedding 
 $\iota : \K \hookrightarrow \R$ with $\iota (\sqrt{m})$  positive. 
Suppose $\sigma$ is a generator of $\Gal(\K/\Q)$ and put 
$\sigma(P)=(\sigma(x),\sigma(y))$. 
Since $P+\sigma(P)$ is an element of $E_{n,\ta}(\Q)$ and $n$ is 
not  a $\ta$-congruent number, 
$P+\sigma(P)\in T_{n,\ta}(\Q)=\{ \infty, (0,0), (-(r+s)n,0), ((r-s)n,0)\}.$
Hence, one of the following cases necessarily happens:
\begin{enumerate}[\upshape I.] 
\item \underline{$P+\sigma(P)=\infty$}. 
In this case, $\sigma(x)=x$ and $\sigma(y)=-y$. So, $x, y\sqrt{m}$  and hence
$U\sqrt{m}, V\sqrt{m}$ and $W\sqrt{m}$ are rational 
 and  we obtain a $(\K,\ta,n)$-triangle of Type 1.
\item \underline{$P+\sigma(P)=(0,0)$}. We have
$\sigma(x)/x=\sigma(y)/y$, which we denote  by $\alpha$. Then,
$$\sigma(y)^2=\alpha^2y^2 =\alpha^2x^3+2sn\alpha^2x^2-(r^2-s^2)n^2\alpha^2x.$$ 
Since $\sigma(P)$ is a point on $E_{n,\ta}$, we get
$$\sigma(y)^2= \sigma(x)^3+2sn\sigma(x)^2-(r^2-s^2)n^2\sigma(x)
=\alpha^3x^3+2sn\alpha^2x^2-(r^2-s^2)n^2\alpha x.$$
Clearly, $\alpha\neq 0, 1$ and $x\neq 0$, which implies 
$x\sigma(x)=\alpha x^2=-(r^2-s^2)n^2.$
Therefore,
 $$V={x(x+2sn+\sigma(x))}/{|y|},\ W\sqrt{m}={x(x-\sigma(x))\sqrt{m}}/{|y|}.$$
Since $x/y=\sigma(x/y)$ and $x>[(r+s)+\sqrt{2r(r-s)}]n$, then $x/|y|$ is rational and
hence $U=2rnx/|y|$, $V$ and $W\sqrt{m}$ are rational, which gives
 a $(\K,\ta,n)$-triangle  of Type 2.
\item \underline{$P+\sigma(P)=((r-s)n,0)$}. We have
$\sigma(x-(r-s)n)/(x-(r-s)n)=\sigma(y)/y,$
which we denote  by $\beta$. Put $z=x-(r-s)n$. Then, 
$$\sigma(y)^2=\beta^2[z^3+(3r-s)nz^2+2r(r-s)n^2z].$$
Since $\sigma(P)$ is a point on $E_{n,\ta}$, we get 
$$\sigma(y)^2=\beta^3z^3+(3r-s)n\beta^2z^2+2r(r-s)n^2\beta z.$$
Now $\beta \neq 0, 1$ and $z\neq 0$, which implies
$\beta z^2=2r(r-s)n^2.$
Substituting this equation and $x=z+(r-s)n$ in $U$, $V$ and $W$, we obtain 
$$U= \frac{z(\sigma(z)+2rn)}{|y|}, \ V=\frac{z(z+2rn)}{|y|}, \
 W=\frac{z(z+2(r-s)n+ \sigma(z))}{|y|}.$$
Since $z/y=\sigma(z/y)$ and $z>0$, then $z/|y|$ and hence $W$ is rational and $\sigma(U)=V$. This time we obtain 
a $(\K,\ta,n)$-triangle of Type 3.
\item \underline{$P+\sigma(P)=(-(r+s)n,0)$}. Put $w=x+(r+s)n$. As in  Case
III,  $w/|y|$ and 
$$ W={w(w-2(r+s)n+ \sigma(w))}/{|y|}$$
 are rational and   $\sigma(U)=-V$, where
$$U= {w(2rn-\sigma(w))}/{|y|},\   V={w(w-2rn)}/{|y|}.$$
Therefore, we obtain a $(\K,\ta,n)$-triangle of Type 4.
\end{enumerate}
\end{proof}
\begin{proof}[\bf{Proof of Proposition \ref{L1}}]
If we suppose that there is a $(\K,\ta,n)$-triangle of Type 2, say 
$(U,V,W)_{\ta}=(u,v,w\sqrt{m})$ with $u,v,w \in \Q^{+}$, then $(x,y,z)=(ru-sv,v,mrw)$ 
is  a non-zero solution of the equation
\begin{equation}\label{eq1}
z^2=mx^2+m(r^2-s^2)y^2.
\end{equation} 
And, if there is a $(\K,\ta,n)$-triangle of Type 3, say $(U,V,W)_{\ta}=(u-v\sqrt{m}, u+v\sqrt{m}, w)$ 
such that  $\sigma(U)=V$, then $(x,y,z)=(u,v,rw)$ is  a non-zero solution of
\begin{equation}\label{eq2}
z^2=2r(r-s)x^2+2mr(r+s)y^2.
\end{equation} 
Similarly, if $(U,V,W)_{\ta}=(-u+v\sqrt{m}, u+v\sqrt{m}, w)$ is a $(\K,\ta,n)$-triangle of Type 4 
such that $\sigma(U)=-V$, then $(x,y,z)=(u,v,rw)$ satisfies 
\begin{equation}\label{eq3}
z^2=2r(r+s)x^2+2mr(r-s)y^2.
\end{equation} 
By the Hasse local-global principle, the equations (\ref{eq1}), (\ref{eq2}) and (\ref{eq3}) have solutions in $\Q$ if and only if  they have a solution in $\Q_p$ for every prime $p$, where $\Q_p$ is the field of $p$-adic numbers.
We assume that $A=\sqf (r^2-s^2)$, and for a prime $p$  the pair $(m, A)$ 
($(m, B)$, and $(m, C)$, resp.) can  be written as 
$(p^\alpha a, p^\beta b )$, where $\alpha, \beta \in \{0,1 \}$ and $\gcd(p,a,b)=1$. 
Then, using Hilbert symbols {\cite[Theorem 1, III]{serr}}, the equations (\ref{eq1}), (\ref{eq2}) and (\ref{eq3}) have solutions 
in $\Q_2$ if and only if one of the following cases happens:
\begin{enumerate}[\upshape  i)] 
 \item $(\alpha, \beta)=(0,0)$ and $(a,b) \mathop {\hspace{.2mm}\not  \equiv} \limits^{4 } (3,3)$;
 \item $(\alpha, \beta)=(0,1)$ and $(a,b) \mathop {\hspace{.2mm}\not  \equiv} \limits^{8 }  
 (3,1), (3,5), (7,5), (7,7)$;
 \item $(\alpha, \beta)=(1,0)$ and $(a,b) \mathop {\hspace{.2mm}\not  \equiv} \limits^{8 } 
 (1,3), (1,5), (3,5), (3,7), (5,3), (5,7),(7,3), (7,7)$;
 \item  $(\alpha, \beta)=(1,1)$ and  $(a,b) \mathop {\hspace{.2mm}\not  \equiv} \limits^{8 } 
 (1,3), (1,5), (3,1), (3,3),(5,1), (5,7),(7,5), (7,7)$.
 \end{enumerate}
Also, the equations (\ref{eq1}), (\ref{eq2}) and (\ref{eq3}) have solutions in $\Q_p$ with $p \mathop   \equiv \limits^{4 } 1$  if and only if one of the following happens:
\begin{enumerate}[\upshape  i)] 
 \item $(\alpha, \beta)=(0,1)$ and $\big(\frac{a}{p}\big)=1$;
 \item $(\alpha, \beta)=(1,0)$ and $\big(\frac{b}{p} \big)=1$;
 \item $(\alpha, \beta)=(1,1)$ and $\big(\frac{a}{p} \big)\big(\frac{b}{p} \big)=1$. 
 \end{enumerate}
\end{proof}
\begin{proof}[\bf{Proof of Theorem \ref{th3}}]
\textit{\textbf{Case 1}}. $n$ and $mn$ are $(\Q,\ta)$-congruent numbers. 
Consider the $(\Q,\ta,n)$-triangle $(U_1,V_1,W_1)_\ta$ and the  
$(\Q,\ta,mn)$-triangle $(U_2,V_2,W_2)_\ta$, where
$$0<U_1\leq V_1<W_1, \ 2rn=U_1V_1,\ U_1^2+V_1^2-\frac{2sU_1V_1}{r}=W_1^2,$$
$$0<U_2\leq V_2<W_2, \ 2rmn=U_2V_2,\ U_2^2+V_2^2-\frac{2sU_2V_2}{r}=W_2^2.$$
Hence, $(U_2/\sqrt{m},V_2/\sqrt{m},W_2/\sqrt{m})_\ta$ is 
a $(\K,\ta,n)$-triangle.
Recall the maps $\varphi$ and 
$\psi$ in the proof of Theorem \ref{th2} and  put 
$$P=(u,v)=\varphi((U_1,V_1,W_1))+\varphi((U_2/\sqrt{m},V_2/\sqrt{m},
W_2/\sqrt{m})).$$ 
Then the  additive law on $E_{n,\ta}(\K)$ implies $u= a+b\sqrt{m}$, where
$$a= \frac{m^3W_1^2(V_1^2-U_1^2)^2+W_2^2(V_2^2-U_2^2)^2 }{4m(W_2^2-mW_1^2)^2} - 
\big(\frac{W_1^2}{4} + \frac{W_2^2}{4m} +2sn \big)>0,$$
$$b= -\frac{W_1W_2(V_1^2-U_1^2)(V_2^2-U_2^2)\sqrt{m}}{2(W_2^2-mW_1^2)^2}.$$
 We may assume $v\geq 0$.
Since $(u,v)\in T$, then $\psi((u,v))\in S$ which indicates the 
sides of a $(\K,\ta,n)$-triangle $(U,V,W)_\ta$.
In fact, if we suppose 
$U=u_1+u_2\sqrt{m},  V=v_1+v_2\sqrt{m}$ and $ W=w_1+w_2\sqrt{m},$
where $u_1,u_2,v_1,v_2,w_1,w_2$ are rational, then 
$$w_1=\pm \sqrt{2(a \pm \sqrt{a^2- mb^2})},  \ \ w_2=\frac{2b}{w_1},$$
and
$$U=(\alpha_1 -\beta_1 ) + (\alpha_2 -\beta_2 )\sqrt{m},\ \ 
V=(\alpha_1 +\beta_1 ) + (\alpha_2 +\beta_2 )\sqrt{m},$$
where 
$$\alpha_1=\pm \sqrt{\frac{(a+(r+s)n) \pm \sqrt{(a+(r+s)n)^2 -mb^2}}{2}}, 
\ \ \alpha_2 =\frac{b}{2\alpha_1},$$
 $$\beta_1=\pm \sqrt{\frac{(a-(r-s)n) \pm \sqrt{(a-(r-s)n)^2 -mb^2}}{2}}, 
 \ \ \beta_2 =\frac{b}{2\beta_1}.$$
\par
Conversely, suppose to the contrary that  $n$ or $mn$ is not  
 $\ta$-congruent over $\Q$. First,
 assume  $n$ is not    $\ta$-congruent over $\Q$  but 
$mn$ is $\ta$-congruent over $\Q$. By Theorem \ref{th2} (1), 
there is no   $(\K,\ta, n)$-triangle $(U,V,W)_\ta$ satisfying 
the conditions $0< U\leq V <W$, $W\not \in \Q$
and $W\sqrt{m} \not \in \Q$. 
\par
\textit{\textbf{Case 2}}. $mn$ is not $\ta$-congruent over $\Q$
but $n$ is  $(\K,\ta)$-congruent. Let  $(U,V,W)_\ta$ 
denotes the sides of the corresponding $(\K,\ta, n)$-triangle.
Multiplying the three sides by $\sqrt{m}$, we get the 
$(\K,\ta, mn)$-triangle $(U\sqrt{m},V\sqrt{m},W\sqrt{m})_\ta$. 
For the  positive integer $mn$, we define the map $\varphi'$ 
 in the same way as  $\varphi$. Put 
 $$2P'=\varphi'((U\sqrt{m},V\sqrt{m},W\sqrt{m}))$$ for some point 
 $P' \in E_{mn,\ta}(\K)$.
For the generator $\sigma$ of $\Gal(\K/\Q)$, since 
$P'+\sigma(P')$ is an element in $E_{mn,\ta}(\Q)$ and $mn$ is 
not  $\ta$-congruent over $\Q$, we have 
$$P'+\sigma(P')\in T_{mn,\ta}(\Q)=\{ \infty, (0,0), (-(r+s)mn,0), 
((r-s)mn,0)\}. $$
Therefore, by the same way as in the proof of Theorem \ref{th2}, 
one of the following cases necessarily happens:
\begin{enumerate}[\upshape {Type} 1.]
 \item $U $, $V $, $W \in \Q$;
 \item $U \sqrt{m}$, $V \sqrt{m}$, $W \in \Q$;
 \item $U$, $V \in K\backslash \Q$ such that $\sigma(U)=V$, $W \sqrt{m}\in \Q$; 
 \item $U$, $V \in K\backslash \Q$ such that $\sigma(U)=-V$, $W \sqrt{m}\in \Q$.
 \end{enumerate}
Hence, there is no $(\K,\ta, n)$-triangle $(U,V,W)_\ta$ 
with $W\not \in \Q$  and $W\sqrt{m}\not \in \Q$.
\par
\textit{\textbf{Case 3}}. Both $n$ and $mn$ are not  $\ta$-congruent numbers over $\Q$,
   where $mn\neq 2,3,6$. If $m\neq \sqf(2r(r-s))$, 
  by Theorem \ref{th1},
$n$ is not  $(\K,\ta,n)$-congruent. If $m=\sqf(2r(r-s))$  
and $n$ is   
$(\K,\ta,n)$-congruent, we have $U=V$ for all  $(\K,\ta,n)$-triangles 
$(U,V,W)_\ta$. Hence, there is no any $(\K,\ta,n)$-triangle  
$(U,V,W)_\ta$ with $W\not \in \Q$ and $W\sqrt{m}\not \in \Q$.
We have completed the proof of Theorem \ref{th3}.
\end{proof}
\section{Examples}
In this section, we  give some examples of $(\K,\ta)$-congruent numbers 
and verify all four types of $(\K,\ta,n)$-triangles
in  Theorem \ref{th2} in the  cases $\ta= \pi/3, 2\pi/3.$  
Given $n$, let $(U,V,W)_{\ta}$ be a  $(\K,\ta,n)$-triangle. 
Then, we have
$$0<U\leq V< W, \ \    UV=2rn, \ \ W^2=U^2+V^2-\frac{2s}{r}UV. $$
For any $(U,V,W)_{\ta}$,  
$\varphi((U,V,W))= ( {W^2}/{4},{W(V^2-U^2)}/{8})$
is a point of $2E_{n,\ta}(\K)\backslash \{\infty\}.$
Also, for any point $(u,v) \in 2E_{n,\ta}(\K)\backslash \{\infty\},$
 $$\psi \big((u,v)\big)=\big((\sqrt{u+(r+s)n} - \sqrt{u-(r-s)n},
\sqrt{u+(r+s)n} + \sqrt{u-(r-s)n},2\sqrt{u}) \big). $$
\par In our computations we have used  Cremona's MWrank program \cite{crem} 
and the number theoretic Pari software \cite{pari}. 

\subsection*{{\bf I) Case $\ta =\pi/3$}}
In this case, we have $r=2$, $s=1$, and $\alpha_{\ta}=\sqrt{3}$, and hence the area 
of any $(\K,\pi/3,n)$-triangle  is $n\sqrt{3}.$
\begin{example}\rm
\label{ex1}
Take $n=3$ and $m=13$. We have the following $(\Q(\sqrt{13}),\pi/3 ,3)$-triangles 
of types 1, 2, 3 and 4 in Theorem \ref{th1} and the corresponding 
points in the set  $2E_{3,\pi/3}(\Q(\sqrt{13}))\setminus \{\infty\}.$ 
\begin{enumerate}[\upshape Type 1.]
\item 
 An easy computing shows that the rank of $E_{39,\pi/3}(\Q)$ is 2, and the   generators of the 
 group  are $P_1=[-9,-216]$ and $P_2=[75,-720]$. 
We have
$$2P_1=[1894/16, -91805/64] \in 2E_{39,\ta}(\Q)\backslash \{\infty\}.$$  
Now, using the map $\varphi$ and $\psi$, defined in the proof of the Theorem 1.1 we get a rational $\pi/3$-triangle $\big (13/2, 24, 43/2  \big)$ with area $39$,
which gives the following 
$(\Q(\sqrt{13}),\pi/3,3)$-triangle of Type 1:
$$(U,V,W)_{\pi/3}= ({\sqrt{13}}/{2}, {24\sqrt{13}}/{13}, 
{43\sqrt{13}}/{26})$$
which corresponds to  the following point
 $Q= ({1894}/{208}, {91805\sqrt{13}}/{416})$.
\item We have a  $(\Q(\sqrt{13}),\pi/3,3)$-triangle  
$(U,V,W)_{\pi/3}=(3,4,\sqrt{13})$      of type 2 with the 
corresponding point 
$Q= ( {13}/{4}, {7\sqrt{13}}/{8}  ) .$
\item
Let $U=u-v\sqrt{13},$  $V=u+v\sqrt{13}$ and $W=w,$ where $u,v,w 
\in \Q \setminus \{0\}.$ 
Then  the pair $( u,v)$ satisfies the equation 
$u^2-13v^2=12.$
An easy solution of this equation is  $(u_0,v_0)=(5,1)$. Parametrizing  $u$ and $v$ in terms of 
$t \in \Q $  we obtain  $u= {-5t^2+26t-65}/{t^2-13} $ and $ v={t^2-10t+13}/{t^2-13}.$
By putting  these  into $w^2=u^2 +39v^2$  and taking  $t=13/4$ one can see that $w^2=u^2 +39v^2$ is a square in $\Q$. So, we obtain 
   $(U,V,W)_{\pi/3}= ( {41-11\sqrt{13}}/{3}, {41+11\sqrt{13}}/{3},{80}/{3}), $ with
   $(\Q(\sqrt{13}),\pi/3,3)$-triangle of type 3 with the
 corresponding point  $Q=( {1600}/{3}, {18040\sqrt{13}}/{9}).$
\item Let  $U=-u+v\sqrt{13},$  $V=u+v\sqrt{13}$ and $W=w,$ 
where $u,v,w \in \Q \setminus \{0\}.$ 
Then  the pair $( u,v)$ satisfies   $13v^2-u^2=12$ with a solution 
$(u_0,v_0)=(1,1)$. A similar discussion as in the previous step, taking  $t=8$,
leads us to a  $(\Q(\sqrt{13}),\pi/3,3)$-triangle 
 of Type 4,  with the corresponding point 
$Q=( {24964}/{51}, {1002352\sqrt{13}}/{51}  ).$
\end{enumerate}
\end {example}
\begin{example}
\label{ex2}
\rm
Let $n=11$ and $m=5$. One can see that $n$ is $\pi/3$-congruent 
over $\Q$ and there is a  $(\Q, \pi/3, 11)$-triangle  $(U_1,V_1,W_1)=( 55/12, 48/5,499/60 ).$
Also, $nm=55$ is $ \pi/3$-congruent  over $\Q$ and  $(U_2,V_2,W_2)=( 8,55/2,49/2 )$
is a rational $\pi/3$-triangle with area $11\sqrt{3}$.
Dividing its sides by $\sqrt{5}$, we obtain a $(\Q( \sqrt{5}), 
\pi/3, 11)$-triangle
$$(U_2/\sqrt{5},V_2/\sqrt{5},W_2/\sqrt{5})=( {8\sqrt{5}}/{5}, 
{11\sqrt{5}}/{2},{49\sqrt{5}}/{10}).$$
Now, a calculations as in   the  proof  of  Theorem \ref{th3} leads to a  
 $(\Q( \sqrt{5}), \pi/3, 11)$-triangle 
$$(U,V,W)=\big(\frac{1}{310}(1470+499\sqrt{5}),\frac{88}{5909} 
(1470-499\sqrt{5}),
\frac{1}{183179} ( 4145193-12554399\sqrt{5}) \big)$$
 satisfying in   Theorem \ref{th3}.
\end {example}
\subsection*{{\bf II) Case $\ta =2\pi/3$}}
In this case, we have $r=2$, $s=-1$, and $\alpha_{\ta}=\sqrt{3}.$ So, 
as in the case I, the area of any $(\K,2\pi/3,n)$-triangle  is $n\sqrt{3}$.
\begin{example}\rm
Take $n=17$ and $m=13$. By a similar way as in Example \ref{ex1}, we find 
the  following $(\Q(\sqrt{13}),2\pi/3 ,17)$-triangles with area $17\sqrt{13}$ of types 1, 2, 3 and 
4 preceding by  their  corresponding points in $2E_{17,2\pi/3}(\Q(\sqrt{13}))\setminus \{\infty\}.$ 
\begin{enumerate}[\upshape Type 1.]
\item
$(U,V,W)_{2\pi/3}= ( {17\sqrt{13}}/{26},8\sqrt{13},{217\sqrt{13}}/{26}),$ \\
$Q= ({47089}/{16},{9325575\sqrt{13}}/{10816})$;
\item  $(U,V,W)_{2\pi/3}= (1,68,19\sqrt{13}),$ $ Q=(13/4, 7\sqrt{13}/8); $
\item
$(U,V,W)_{2\pi/3}= ( 9-\sqrt{13},  9+\sqrt{13}, 16  ),$   $ Q=( 64,72\sqrt{13}  );$
\item 
$(U,V,W)_{2\pi/3}= ( {-5+7\sqrt{13}}/{3}, {5+7\sqrt{13}}/{3}, {44}/{3}  ),
Q=( {484}/{9}, {770 \sqrt{13}}/{27} ).$
\end{enumerate}
\end {example}
\begin{example}\rm
Let $n=19$ and $m=6.$ Then $19$ is a
$2\pi/3$-congruent number over $\Q$ and there is a 
$(\Q,2\pi/3,6)$-triangle $(U_1,V_1,W_1)=( {544}/{105}, {1995}/{136}, {254659}/{14280} )$ 
with area $19\sqrt{3}$.
Also, the integer $nm=114$ is a $ 2\pi/3$-congruent number over $\Q$
and $(U_2,V_2,W_2)=( 5, {912}/{10}, {469}/{5} )$ is a   $2\pi/3$-triangle with area $114\sqrt{3}$
from which we obtain a $(\Q( \sqrt{6}), 2\pi/3, 19)$-triangle 
$$( {5\sqrt{6}}/{6}, {76\sqrt{6}}/{5}, {469\sqrt{6}}/{30} ).$$
By a similar methods as in  Example \ref{ex2}, one can find  a
 $(\Q( \sqrt{6}), 2\pi/3, 19)$-triangle 
\begin{multline*} 
(U,V,W)_{2\pi/3}= \big((25449816+4838521\sqrt{6} )/4683550, \\
{20}( 4145193-12554399\sqrt{6} )/{28499829},\\
{7} (3589965612532-2573211605723\sqrt{6})/{1170880474675} \big),
\end{multline*}
satisfying  Theorem \ref{th3}. 
\end {example}

\end{document}